\date{today}

\documentclass[10pt, letterpaper]{amsart}

\usepackage{latexsym, amsmath, amstext, amssymb, amsfonts, amscd, bm, array, multirow, amsbsy, mathrsfs}
\usepackage{amsthm}
\usepackage{t1enc}
\usepackage[mathscr]{eucal}
\usepackage{indentfirst}
\usepackage{pb-diagram}
\usepackage{graphicx}
\usepackage{fancyhdr}
\usepackage{fancybox}
\usepackage{enumerate}
\usepackage{color}
\usepackage[all]{xy}
\usepackage{hyperref}
\usepackage{tikz}
\usepackage{xparse}
\hypersetup{colorlinks=false,pdfborderstyle={/S/U/W 0}}
\usetikzlibrary{matrix}

\usepackage{url}
\usepackage[sort&compress,comma]{natbib}
\bibpunct{[}{]}{,}{n}{}{,}
\theoremstyle{plain}
\newtheorem{thm}{Theorem}[section]

\newtheorem{prop}[thm]{Proposition}

\newtheorem{lemma}[thm]{Lemma}
\newtheorem{cor}[thm]{Corollary}

\theoremstyle{definition}
\newtheorem{definition}[thm]{Definition}

\theoremstyle{remark}
\newtheorem{remark}[thm]{Remark}

\renewcommand{\a} \alpha
\renewcommand{\b} \beta

\newcommand{\Map}{\mathrm{Map}}

\newcommand{\G}{\mathrm{G}}

\newcommand{\frg}{\mathfrak{g}}

\newcommand{\cA}{\mathcal{A}}
\newcommand{\cB}{\mathcal{B}}

\newcommand{\cE}{\mathcal{E}}
\newcommand{\cF}{\mathcal{F}}
\newcommand{\cG}{\mathcal{G}}

\newcommand{\cM}{\mathcal{M}}

\newcommand{\E}{\mathrm{E}}

\newcommand{\RR}{\mathbb{R}}
\newcommand{\CC}{\mathbb{C}}
\newcommand{\HH}{\mathbb{H}}

\newcommand{\ZZ}{\mathbb{Z}}

\newcommand{\too}{\longrightarrow}
\newcommand{\x}{\times}
\newcommand{\ox}{\otimes}

\newcommand{\la}{\langle}
\newcommand{\ra}{\rangle}
\newcommand{\frM}{{\frak M}}

\newcommand{\Id}{\mathrm{Id}}

\newcommand{\odd}{\textit{odd}}

\newcommand{\frD}{\mathfrak{D}}

\newcommand{\Hol}{\mathrm{Hol}}
\newcommand{\frz}{\mathfrak{z}}
\newcommand{\fro}{\mathfrak{o}}

\DeclareMathOperator{\Ad}{Ad}
 
\DeclareMathOperator{\im}{im}

\DeclareMathOperator{\coker}{coker}
\DeclareMathOperator{\Hom}{Hom}

\DeclareMathOperator{\End}{End} 

\DeclareMathOperator{\GL}{GL}
\DeclareMathOperator{\SO}{SO}
\DeclareMathOperator{\SU}{SU}
\DeclareMathOperator{\Spin}{Spin}
\DeclareMathOperator{\U}{U}

\DeclareMathOperator{\Det}{Det}
\DeclareMathOperator{\ad}{ad}

\DeclareMathOperator{\tr}{tr}

\begin{document}

\title{Orientability of the moduli space of $\Spin(7)$-instantons}

\author[V. Mu\~{n}oz]{Vicente Mu\~{n}oz}
\address{Facultad de Ciencias Matem\'aticas, Universidad
Complutense de Madrid, Plaza de Ciencias 3, 28040 Madrid, Spain}
\email{vicente.munoz@mat.ucm.es}

\author[C. S. Shahbazi]{C. S. Shahbazi}
\address{Institut f\"ur Theoretische Physik, Leibniz Universit\"at Hannover, Appelstra{\ss}e 2, 30167 Hannover, Germany}

\email{carlos.shahbazi@itp.uni-hannover.de}

\thanks{2010 MSC. Primary:  53C38. Secondary: 53C07, 53C25.}
\keywords{$\Spin(7)$-instanton, moduli space, $\Spin(7)$-structure}

\begin{abstract}
Let $(M,\Omega)$ be a closed $8$-dimensional manifold equipped with a generically 
non-integrable $\Spin(7)$-structure $\Omega$. We prove that if 
$\Hom(H^{3}(M,\mathbb{Z}), \mathbb{Z}_{2}) = 0$ then the moduli space of irreducible 
$\Spin(7)$-instantons on $(M,\Omega)$ with gauge group $\SU(r)$, $r\geq 2$, is orientable. 
\end{abstract}

\maketitle

\noindent \emph{In commemoration of the $60^{th}$ birthday of Prof.\ Simon Donaldson, with our utmost gratitude for all we have learnt from him.}


\section{Introduction} \label{sec:intro}

Higher-dimensional gauge theory is a proposal appearing in the influential work of Donaldson and  Thomas \cite{DT1} which suggests studying a natural higher-dimensional version of the four-dimensional instanton equations that turned out to be so fundamental in the understanding of four-manifold topology. These higher dimensional instanton-like equations exist in the presence of appropriate geometric structures. The long term hope is that the study of their moduli space of solutions will shade light into the classification problem of the underlying manifold equipped with the corresponding geometric structure. It should be noted that the instanton equations proposed in \cite{DT1} had previously appeared in the physics literature, see for example \cite{CDFN,Ward}, and therefore they are also well motivated from a physical point of view.
 
In \cite{MunozIII}, the authors initiated the construction of the moduli space of $\Spin(7)$-instantons on a closed $8$-manifold $M$ equipped with a possibly non-integrable $\Spin(7)$-structure, proving transversality properties of the moduli space under suitable perturbations of the equations and the $\Spin(7)$-structure.

To be more precise, let $M$ be a closed $8$-dimensional manifold. A $\Spin(7)$-structure is a $4$-form $\Omega\in \Omega^4(M)$ which at every point $p\in M$, it can be written as
\begin{align*} 
\Omega_{p}  =& \, dx_{1234} - dx_{1278} - dx_{1638} - dx_{1674} + dx_{1526} + dx_{1537} + dx_{1548} \nonumber\\ 
&+ dx_{5678} - dx_{5634} - dx_{5274} - dx_{5238} + dx_{3748} + dx_{2648} + dx_{2637}\, ,
 \end{align*}
for suitable coordinates $(x_1,\ldots, x_8)$, where $dx_{abcd}$, $a, b, c, d =1,\hdots, 8$, stands for $dx_{a}\wedge dx_{b}\wedge dx_{c}\wedge dx_{d}$. The stabilizer of $\Omega_p$ is a subgroup of $\GL(8,\RR)$ isomorphic to $\Spin(7)$, which is the double cover of $\SO(7)$. There is a natural inclusion $\Spin(7)<\SO(8)$, which gives $M$ an orientation and a Riemannian metric $g$. Let $\nabla$ be the Levi-Civita connection of $g$. Then the holonomy of $g$ is contained in $\Spin(7)<\SO(8)$ when $\nabla \Omega=0$, which is equivalent to $d\,\Omega=0$ by \cite{Fernandez-Gray}. In this case the $\Spin(7)$-structure
is called \emph{integrable}. If $d\,\Omega\neq 0$ then we say that $\Omega$ is a non-integrable $\Spin(7)$-structure.

If $M$ is an $8$-manifold with a $\Spin(7)$-structure, then the inclusion $\Spin(7)<\Spin(8)$ says that $M$ is 
a spin manifold, that is, it admits a $\Spin(8)$-structure (the frame bundle $\mathrm{Fr}_{\SO(8)}(TM)$ admits
a lifting under the double cover $\Spin(8)\to \SO(8)$). Associated to $\Spin(8)$, there are two irreducible spinor
representations $S^\pm$, which are of dimension $8$. Therefore $M$ has two spinor vector bundles $S^\pm\to M$
of rank $8$. The group $\Spin(7)$ can also be defined as the stabilizer of a unit-norm element $\eta$ in $S^+$,
so a $\Spin(7)$-structure is equivalent to the choice of such $\eta \in\Gamma(S^+)$.

Let $\G$ be a compact Lie group and let $P\to M$ be a principal $\G$-bundle. 
In terms of $\Omega$, the $\Spin(7)$-instanton equation for a connection $A$ on $P$ is given by
  \begin{equation}\label{eqn:main}
  \ast F_{A} = -F_{A}\wedge\Omega\, ,
  \end{equation}
where $F_{A}$ is the curvature associated to $A$. We will refer to solutions of this equation as $\Spin(7)$-instantons. 
We are interested in studying the \emph{moduli space} of $\Spin(7)$-instantons, that is, solutions of the $\Spin(7)$-instanton equation
modulo gauge transformations (automorphisms of the bundle $P$). 
The $\Spin(7)$-instanton equation modulo gauge transformations is elliptic regardless of
the integrability of the underlying $\Spin(7)$-structure \cite{Carrion}.  
This makes the study of moduli spaces of $\Spin(7)$-instantons very similar to the moduli spaces of anti-self-dual
instantons for Riemmanian $4$-dimensional manifolds, which is the central object of low-dimensional gauge theory \cite{DKbook}.

In the work \cite{MunozIII},  the authors studied various ways to perturb the $\Spin(7)$-instanton equations as
to get smooth moduli spaces of the expected dimension, on the locus of irreducible connections. Let $\cA$ be the space
of $\G$-connections on $P$, and let $\cA^*$ be the subspace of irreducible connections. We denote by $\cG$ the gauge  
group, that is, the group of automorphisms of $P$. Then the space $\cB^*=\cA^*/\cG$ is the configuration space,
the space of irreducible connections modulo isomorphism. The moduli space of irreducible $\Spin(7)$-instantons $\cM^*$ sits naturally as
a subspace $\cM^*\subset \cB^*$. For a suitable type of perturbation $\varpi\in \Pi$ of the equation (\ref{eqn:main}),
where $\Pi$ is some space of parameters as the ones defined in \cite{MunozIII},
we have a \emph{perturbed moduli space} $\cM^*_\varpi \subset\cB^*$. The main object of the present paper is
to study the orientability of $\cM^*_\varpi$, for the previous classes of perturbations such that the moduli space is regular. 
We do not need to enter into the nature of these perturbations (for which we refer to \cite{MunozIII}), since 
it is not necessary to deal explicitly with them. In this work we prove the following result.

\begin{thm} \label{thm:main}
Consider any gauge group $\G=\SU(r)$, $r\geq 2$. 
Let $\varpi\in \Pi$ be a suitable perturbation so that the perturbed moduli space of $\Spin(7)$-instantons 
$\cM^*_\varpi$ is regular. If $\Hom(H^{3}(M,\mathbb{Z}), \mathbb{Z}_{2}) = 0$ then $\cM^*_\varpi$ is orientable.
\end{thm}

Note that the condition  $\Hom(H^{3}(M,\mathbb{Z}), \mathbb{Z}_{2}) = 0$ of Theorem \ref{thm:main} is equivalent
to the cardinality $| H^{3}(M,\mathbb{Z})|$ being finite and odd. 

To prove Theorem \ref{thm:main}, we start by describing the orientation line bundle of the moduli 
space $\cM^*_\varpi$ as the restriction to $\cM^*_\varpi \subset\cB^*$
of the so-called orientation bundle $\fro$ over the configuration space 
$\cB^*$. The orientation bundle is 
the determinant line bundle associated to the Dirac operator, that is, for $A\in \cB^*$, we have $\fro_A=\Det(\frD_A)$, where $\frD_A$ is the Dirac operator coupled with the connection $A$. Therefore
the orientation is controlled by the orientation class $W=w_1(\Det(\frD)):\pi_1(\cB^*)\to \ZZ_2$.

To determine $\pi_1(\cB^*)$, we consider as gauge group the exceptional Lie group $\G=\E_8$.
We use the group $\E_8$ as a tool to study the orientability of the moduli spaces of instantons for the groups 
$\G=\SU(r)$. This may seem a strange choice at first sight, but it is a very convenient group due to its homotopical properties. 
Also it fits well with the long tradition of looking at instantons with arbitrary structure Lie groups \cite{AtiyahHitchin}, 
and in particular, including exceptional Lie groups as very relevant cases. In addition, the group $\E_{8}$ plays a 
fundamental role in the formulation of one of the two types of Heterotic string theories, whose gauge sector 
indeed contains instantons with gauge group $\E_{8}$ (see \cite{BBS}).

For the gauge group $\G=\E_8$, the configuration space has fundamental group $\pi_1(\cB^*)\cong H^3(M,\ZZ)$ (see Proposition 
\ref{prop:E8orientability}). Therefore, the orientation class $W$ vanishes when $\Hom(H^{3}(M,\mathbb{Z}), \mathbb{Z}_{2}) = 0$. Our
last step is to move from $\E_8$ to the gauge groups $\G=\SU(r)$. We first use the natural inclusion $\SU(9)/\ZZ_3 <\E_8$, to 
deduce the orientability for the moduli spaces of $\Spin(7)$-instantons with gauge group $\SU(9)/\ZZ_3$ (see Propostion 
\ref{prop:thm:orientability1}). This readily implies
the same result for the gauge group $\SU(9)$. From here we deduce the result for all other gauge groups $\SU(r)$, $2\leq r\leq 8$, by an
easy downward induction. The result for the gauge group $\SU(r)$, $r\geq 10$, follows from the fact that the inclusion 
$\SU(9)<\SU(r)$ is a homotopy equalence up to degree $18$. This completes the proof of Theorem \ref{thm:main}.

We remark that the results of Cao and Leung (Theorem 2.1 in \cite{CaoConanII})
can easily be shown to imply that the moduli space of $\Spin(7)$-instantons on a closed $8$-manifold manifold $M$ is orientable when
$\Hom(H^{\odd}(M,\mathbb{Z}),\mathbb{Z}_{2}) = 0$, 
and the gauge group is $\G=\SU(r)$, $r \geq 2$. Therefore, Theorem \ref{thm:main} generalizes the results of \cite{CaoConanII} by 
dropping any assumption on the cohomology groups $H^{1}(M,\mathbb{Z})$ and $H^{7}(M,\mathbb{Z})$ of $M$. 
In particular, we do not require any condition on the fundamental group of the underlying $\Spin(7)$-manifold $M$.

Finally, note that there are manifolds where Theorem \ref{thm:main} can be applied, see Remark \ref{rem:6.3}.

\subsection*{Acknowledgements} 
We are very grateful to the referee for useful comments.
We thank Aleksander Doan, Simon Donaldson, Dominic Joyce and 
Thomas Walpuski for useful conversations. First author partially supported through Project 
MICINN (Spain) MTM2015-63612-P. Second author was partially supported by the German Science Foundation (DFG) Project LE838/13.

\section{The configuration space}
\label{sec:setup}

Let $(M,\Omega)$ be an $8$-dimensional manifold equipped with an $\Spin(7)$-structure $\Omega$ and let $P$ be a principal $\G$-bundle over $M$, where $\G$ is a compact semi-simple Lie group with Lie algebra $\frg$. 
We denote by $B\G$ the classifying space of $\G$. Associated to $P$ we consider a complex vector bundle $E = P\times_{\rho} \mathbb{C}^{r}$ of rank $r$, where $\rho:\G \to \GL(r,\CC)$ is an
$r$-dimensional faithful 
complex representation of $\G$. We denote by $\frg_{E}\subset \End(E)$ the endomorphism bundle of $E$ associated to the adjoint bundle of algebras $\mathrm{ad}(P) = P\times_{\Ad} \frg$ of $P$ through $\rho$.

\begin{remark}
We will be mainly interested in the cases $\G = \SU(r)$ and $\G = \E_{8}$, where $\E_{8}$ denotes the unique connected, simply-connected, compact simple Lie group associated to the exceptional Dynkin diagram $\mathfrak{e}_{8}$. The properties of the group $\E_8$ relevant for us appear in Section \ref{sec:E8}.
\end{remark}

We will denote by $\cA$ the space of $\G$-compatible connections on $E$. For
a connection $A\in \cA$, we denote by $F_A\in\Omega^{2}(\frg_{E})$ its curvature. 

The group of gauge transformations  $\cG$ is defined as the group of all smooth automorphisms of $E$ or, equivalently, as the space $C^{\infty}(\Ad P)$ of all smooth sections of the bundle $\Ad P = P\times_{\Ad} \G$, where $\G$ acts on itself by conjugation. 
An equivalent description is given by $\cG = \Map_{\G}(P, \G)$,
i.e., smooth maps from $P$ to $\G$ which are $\G$-equivariant with respect to the adjoint action of $\G$ on itself. 
We equip $\cG$ with the subspace topology induced by the compact-open topology of $\Map(E,E)$, the space of all smooth maps of $E$ to itself. Equipped with this topology and with the product given by composition, $\cG$ becomes a topological group. 

The center $Z(\cG)$ of $\cG$ is given by $Z(\cG) = \Map(M, Z(\G))$. 
We need to introduce two spaces closely related to the gauge group $\cG$. We define the \emph{reduced gauge group} as $\bar{\cG} = \cG/Z(\G)$, where $Z(\G)$ denotes the center of $\G$.

\begin{remark}
We have $Z(\SU(r))=\mathbb{Z}_{r}$ and $Z(\E_{8}) = \Id$. Hence, for $\G = \E_{8}$ we obtain $\cG = \bar{\cG}$.
\end{remark}

Let $p_0 \in M$ be a fixed base point. We define the \emph{framed gauge group} as
 \begin{equation*}
 \cG_{0} = \left\{ u\in\cG \,\, |\,\, u_{p_0} = \Id\in \End(E_{p_0})\right\}.
 \end{equation*}
We have a fibration
 \begin{equation}\label{eqn:Hurewicz} 
 \cG_0 \to \cG \stackrel{\delta}{\to}  \G,
 \end{equation}
where $\delta(u)=u_{p_0}$.

We define the spaces of equivalence classes of connections and of framed connections as
 \begin{equation} \label{eqn:BG}
 \cB = \cA/ \bar{\cG}\, , \qquad \cB_{0} = \cA/\cG_{0}\, ,
 \end{equation}
which are also called the \emph{configuration spaces}.

We say that a connection $A\in \cA$ is irreducible if its holonomy group $\Hol(A)= \G$, 
and we denote by $\cA^{\ast}\subset \cA$ the space of irreducible connection. We also
denote the spaces of equivalence classes of irreducible connections by $\cB^{\ast} = 
\cA^{\ast}/ \bar{\cG}\subset\cB$, and $\cB_{0}^* = \cA^{\ast}/\cG_{0}\subset \cB_0$.

The reduced gauge group $\bar{\cG}$ and the framed gauge group $\cG_{0}$ act freely on $\cA^{\ast}$ and $\cA$, respectively. The space $\cA$ is an affine space, hence contractible. The subspace $\cA^{\ast}\subset \cA$ is the complement of the space of reducible connections which is an infinite-codimensional subspace. Hence $\cA^{\ast}$ is also contractible.
It follows that we have homotopy equivalences
 \begin{equation}\label{eqn:hom-B-cB} 
 \cB^* \simeq B\bar\cG, \qquad  \cB_0 \simeq  B\cG_0
 \end{equation}
where $B\bar\cG$ and $B\cG_{0}$ denote the classifying spaces of $\bar\cG$ and $\cG_{0}$, respectively.

\begin{lemma}\label{lem:G-sc-noZ}
If $\G$ is a simply-connected compact Lie group and $Z(\G)=\Id$, then $\pi_1(\cB_0)\cong \pi_1(\cB^*)$.
\end{lemma}

\begin{proof}
The fibration (\ref{eqn:Hurewicz}) gives a fibration $ B\cG_0  \to B\cG \to B\G$. 
If $\G$ is simply-connected then $\pi_1(B\G)=\pi_2(B\G)=1$. Therefore
 \begin{equation}\label{eqn:pi1}
 \pi_1(B\cG_0) \cong \pi_1(B\cG).
 \end{equation}
If $Z(\G)=\Id$ then $Z(\cG)=\Id$ and hence $\bar\cG=\cG$. The result follows from (\ref{eqn:hom-B-cB}).
\end{proof}

Let us denote by $\Map^{P}(M,B\G)$ the path-connected component of $\Map(M,B\G)$ containing $P$. Likewise let $\Map^{P}_{\ast}(M,B\G)$ denote the path-connected component of $\Map_{\ast}(M,B\G)$ containing $P$, where $\Map_{\ast}(M,B\G)$ denotes the set of point-based maps from $M$ to $B\G$. 

\begin{prop}[{\cite[Proposition 2.4]{AtiyahBott}}] \label{prop:homotopyB0}
The following weak homotopy equivalences hold
 \begin{align*}
  & B\cG 
  \simeq \Map^{P}(M,B\G)\, , \qquad 
   B\cG_{0} 
 \simeq\Map^{P}_{\ast}(M,B\G).
\end{align*}
\end{prop}

\begin{proof}
  The universal fibration $\G \to E\G  \to B\G$ gives a fibration 
 $$
 \cG=\Map_G(P,G) \to \Map_\G(P,E\G) \to \Map^P(M,B\G).
 $$
The space $\Map_\G(P,E\G)$ is contractible, hence $\Map^P(M,B\G) \simeq B\cG$. The second statement is analogous.
\end{proof}

\noindent
We will need the following result.

\begin{prop}[{\cite[Theorem 6.1]{ConjugacyClasses}}] \label{prop:thm:pahtcomp}
Suppose that the path-components of $\Map_{\ast}(M,B\G)$ all
have the same homotopy type. Then the following homotopy equivalence holds:
\begin{equation*}
\cG_{0} \simeq \Map_{\ast}(M,\G)\, .
\end{equation*}
Furthermore, the homotopy equivalence preserves the multiplicative structures.
\end{prop}

\begin{remark}
\label{remark:COH}
A sufficient condition for the path connected components of $\Map_{\ast}(M,B\G)$ 
to have the same homotopy type is for $M$ to be an associative CoH-space, 
see \cite[Section 9]{Strom} for a detailed exposition of CoH-spaces. This happens for example if $M$ is a suspension.
\end{remark}

\section{The group $\E_{8}$} \label{sec:E8}

Here we recollect some facts on the exceptional group $\E_{8}$ and principal bundles with gauge group $\E_{8}$. 
%
%
By $\E_{8}$ we denote the unique simple and simply connected, compact, real Lie group whose Lie algebra $\mathfrak{e}_{8}$ is a real form of the simple complex exceptional Lie algebra $\mathfrak{e}_{8}^{\mathbb{C}}$ appearing in the Cartan-Killing classification of complex simple Lie algebras. We refer to \cite{Adams,Yokota} for more details. 
The group $\E_{8}$ has real dimension $248$ and embeds as a closed Lie subgroup of the unitary group $\U(248)$. There are two key points that make $\E_{8}$ relevant for the goal of this paper. The first one is that the center 
 \begin{equation*}
 Z(\E_{8}) = \Id
 \end{equation*}
is trivial. The second one is the content of the following classical theorem, proved by Bott and Samelson, on the homotopy of $\E_8$.

\begin{thm}[{\cite[Theorem V]{BottSamelson}}] \label{thm:piE8}
The real exceptional Lie group $\E_{8}$ is connected, simply-connected and its low homotopy groups are
 \begin{equation*}
 \pi_2(\E_8)=0, \qquad \pi_{3}(\E_{8}) = \mathbb{Z}\, , \qquad \pi_{i}(\E_{8}) = 0\, , 
 4\leq i\leq 14\, , \qquad \pi_{15}(\E_{8}) = \mathbb{Z}\, .
\end{equation*}
\end{thm}

We will only need that for $i\leq 8$ the only non-zero homotopy group is $\pi_{3}$, which in fact is forced to be isomorphic to $\mathbb{Z}$ for any compact simple Lie group. The group $\E_{8}$ contains various important subgroups, of which we will be concerned with $\SU(9)/\mathbb{Z}_{3}\subset \E_{8}$. We denote by $\mathfrak{e}_{8}$ the Lie algebra of $\E_{8}$ and by $\mathfrak{e}^{\mathbb{C}}_{8}$ its complexification, which is
a $248$-dimensional simple complex Lie algebra. From the existence of the embedding $\SU(9)/\mathbb{Z}_{3}$ in $\E_{8}$ we get that
 \begin{equation*}
 (\mathfrak{su}(9))^{\mathbb{C}} \cong \mathfrak{sl}(9,\mathbb{C})\subset \mathfrak{e}^{\mathbb{C}}_{8}\, ,
 \end{equation*}
and hence we conclude that $\mathfrak{sl}(9,\mathbb{C})$ is a subalgebra of $e^{\mathbb{C}}_{8}$. Using this subalgebra we can obtain an explicit model for the Lie algebra $\mathfrak{e}^{\mathbb{C}}_{8}$. In order to do this we consider the space of $k$-forms $\Lambda^{k}(\mathbb{C}^{9})$ on $\mathbb{C}^{9}$, where we equip $\mathbb{C}^{9}$ with its canonical inner product $(\cdot,\cdot)$.
We extend the standard inner product to $\Lambda^{k}(\mathbb{C}^{9})$. 
The group $\mathrm{SL}(9,\mathbb{C})$ acts on $\Lambda^{k}(\mathbb{C}^{9})$ in the standard way
 \begin{equation*}
 A\cdot (x_{1}\wedge \ldots \wedge x_{k}) = (A x_{1}\wedge \ldots \wedge A x_{k})\, .
 \end{equation*}

Note that $A \cdot 1 = 1$ and $A\cdot (x_{1}\wedge \cdots \wedge x_{9}) = x_{1}\wedge \cdots \wedge x_{9}$. 
The previous action induces an action of the Lie algebra $\mathfrak{sl}(9,\mathbb{C})$ of $\mathrm{SL}(9,\mathbb{C})$ on $\Lambda^{k}(\mathbb{C}^{9})$, which reads as follows
 \begin{equation*}
 R\cdot (x_{1}\wedge \cdots \wedge x_{k}) = \sum_{j=1}^{k} x_{1}\wedge \cdots \wedge R x_{j}\wedge \cdots\wedge x_{k}\, .
 \end{equation*}
Note that $R\cdot 1 =0$. The following theorem gives an explicit model of the complexification $\mathfrak{e}^{\mathbb{C}}_{8}$ of the real Lie algebra $\mathfrak{e}_{8}$ in terms of $\mathfrak{sl}(9,\mathbb{C})$, $\Lambda^{3}(\mathbb{C}^{9})$ and the action introduced above.

\begin{thm}[{\cite[Theorem 5.11.3]{Yokota}}]
Let us consider the following $248$-dimensional vector space
 \begin{equation} \label{eq:splittinge8C}
 \mathfrak{e}^{\mathbb{C}}_{8} := \mathfrak{sl}(9,\mathbb{C})\oplus \Lambda^{3}(\mathbb{C}^{9})\oplus \Lambda^{3}(\mathbb{C}^{9})\, ,
 \end{equation}
equipped with the Lie bracket
$[R_{1}\oplus x_{1}\oplus y_{1}, R_{2}\oplus x_{2}\oplus y_{2}] = R\oplus x\oplus y$, where
\begin{align*}
R &:= [R_{1},R_{2}] + x_{1}\times y_{2} - x_{2}\times y_{1}\, ,\\
x &:= R_{1}\cdot x_{2} - R_{2}\cdot x_{1} + \ast (y_{1}\wedge y_{2})\, ,\\
y &:= - R^{t}_{1}\cdot y_{2} + R^{t}_{2}\cdot y_{1} - \ast (x_{1}\wedge x_{2})\, .
\end{align*}
Then $(\mathfrak{e}^{\mathbb{C}}_{8}, [\cdot,\cdot])$ is a complex simple Lie algebra of type $\E_{8}$. 
\end{thm}

Given $x, y \in \Lambda^{3}(\mathbb{C}^{9})$, the operation $x\times y \in \mathfrak{sl}(9,\mathbb{C})$ is defined as
\begin{equation*}
(x\times y)(u) = \ast (y \wedge \ast (x\wedge u)) + \frac{2}{3} (x,y) u\, , \quad \forall\, u\in \mathbb{C}^{9}\, .
\end{equation*}
It can be checked that $\mathrm{Tr}(x\times y) = 0$ for all $x,y\in \Lambda^{3}(\mathbb{C}^{9})$, whence $x\times y \in \mathfrak{sl}(9,\mathbb{C})$. The Killing form of $\mathfrak{e}^{\mathbb{C}}_{8}$ in the decomposition of $\mathfrak{e}^{\mathbb{C}}_{8}$ introduced above is explicitly given by
\begin{equation*}
\kappa(R_{1}\oplus x_{1} \oplus y_{1},R_{2} \oplus x_{2} \oplus y_{2}) = 60\, ( \tr(R_{1} R_{2}) + (x_{1},y_{2}) + (x_{2},y_{1}))\, .
\end{equation*}

We denote by $\tau$ the complex conjugation in $\mathfrak{e}^{\mathbb{C}}_{8}$ and define the following complex-conjugate linear transformation of $\mathfrak{e}^{\mathbb{C}}_{8}$
\begin{equation*}
\hat{\tau}(R\oplus x\oplus y) = -\tau^{t} R \oplus -\tau x \oplus -\tau y\, , 
\end{equation*}
for all $R\oplus x\oplus y \in \mathfrak{e}^{\mathbb{C}}_{8}$.

\begin{definition}
The group $\E^{\mathbb{C}}_{8}$ is defined as the automorphism group of the Lie algebra $\mathfrak{e}^{\mathbb{C}}_{8}$
\begin{equation*}
\E^{\mathbb{C}}_{8} := \left\{  A\in \mathrm{GL}(\mathfrak{e}^{\mathbb{C}}_{8},\mathbb{C})\,\, |\,\, 
[A(X_{1}),  A(X_{2})] = A([X_{1},X_{2}])\, , \,\, X_{1}, X_{2} \in \mathfrak{e}^{\mathbb{C}}_{8}   \right\}\, .
\end{equation*}
\end{definition}

Using the previous definition, the real compact and simply-connected Lie group $\E_{8}$ can now be neatly defined as
\begin{equation*}
\E_{8} := \left\{ A\in\E_{8}^{\mathbb{C}}\,\, |\,\, \kappa(A(X_{1}), \hat{\tau} A(X_{2})) = \kappa(X_{1}, \hat{\tau} X_{2})\, , \,\, X_{1}, X_{2} \in \mathfrak{e}^{\mathbb{C}}_{8} \right\} .
\end{equation*}

The polar decomposition of $\E^{\mathbb{C}}_{8}$ is given in terms of $\E_{8}$ by
\begin{equation*}
\E^{\mathbb{C}}_{8} \cong \E_{8}\times \mathbb{R}^{248}\, ,
\end{equation*}
and hence $\E^{\mathbb{C}}_{8}$ and $\E_{8}$ are homotopy equivalent, as expected. We define now the following $\mathbb{C}$-linear transformation $w$ of $\mathfrak{e}^{\mathbb{C}}_{8}$
\begin{equation*}
w(R,x,y) = R\oplus \omega\, x \oplus \omega^{2}\, y\, ,
\end{equation*}
where $\omega = -\frac{1}{2} + \frac{\sqrt{3}}{2} i$, and $R\oplus x\oplus y \in \mathfrak{e}^{\mathbb{C}}_{8}$. It can be seen that $w$ is in fact an element of $\E_{8}$ which satisfies $w^{3} = \Id_{\E_{8}}$. We consider the commutant of the image of $w$ in $\E_{8}$
\begin{equation*}
\E^{w}_{8} := \left\{ A\in \E_{8} \,\, |\,\, w\, A = A\, w\right\}\, .
\end{equation*}

\begin{thm}[{\cite[Theorem 5.11.7]{Yokota}}]
We have an isomorphism of Lie groups $\E^{w}_{8} \cong \SU(9)/\mathbb{Z}_{3}$, 
where $\mathbb{Z}_{3} = \left\{\Id, \omega \,\Id, \omega^{2} \,\Id\right\}$.
\end{thm}

\begin{proof}
The proof of this result relies on the existence of a map $\varphi\colon \SU(9)\to \E^{w}_{8}$ defined for every $A\in \SU(9)$ as
 \begin{equation} \label{eq:su9action}
 \varphi(A)\circ (R\oplus x\oplus y) := A R A^{-1} \oplus A\cdot x \oplus (A^{-1})^t \cdot y\, ,
 \end{equation} 
where $R\oplus x \oplus y\in \mathfrak{e}_{8}$ using the decomposition \eqref{eq:splittinge8C}. It can be seen that $\varphi$ is surjective and furthermore $\ker(\varphi) \cong\mathbb{Z}_{3} = \left\{\Id, \omega \,\Id, \omega^{2} \,\Id\right\}  $, from which the result follows. The map $\varphi$ encodes the way the complex adjoint representation of $\E_{8}$ acts on $\mathfrak{e}^{\mathbb{C}}_{8}$ when restricted to $\mathrm{im}(\varphi)\subset \E_{8}$. For every $A\in \SU(9)$ we have
\begin{equation*}
\Ad_{\varphi(A)}(R,x,y) = \varphi(A)\circ (R,x,y)\, .
\end{equation*}
\end{proof}

Since $\SU(9)/\mathbb{Z}_{3}\subset \E_{8}$, we can consider the restriction to $\SU(9)/\mathbb{Z}_{3}$ of any given representation of $\E_{8}$. The smallest irreducible complex representation of $\E_{8}$ is the adjoint, which is complex 
$248$-dimensional. The explicit decomposition of the complex adjoint of $\E_{8}$ in $\SU(9)/\mathbb{Z}_{3}$ representations corresponds to
 \begin{equation*} \label{eq:decompositione8cII}
 \mathfrak{e}^{\mathbb{C}}_{8} = \mathfrak{sl}(9,\mathbb{C})\oplus \Lambda^{3}(\mathbb{C}^{9})\oplus 
 \Lambda^{3}(\mathbb{C}^{9})^{\ast}\, ,
 \end{equation*}
where $\Lambda^{3}(\mathbb{C}^{9})^{\ast}$ denotes the complex dual representation of $\Lambda^{3}(\mathbb{C}^{9})$. The group $\SU(9)/\mathbb{Z}_{3}$ acts, for all $A\in \SU(9)/\ZZ_3$, through $\varphi$ as prescribed in (\ref{eq:su9action}).

Let us set $\Lambda_{\mathbb{C}} := \Lambda^{3}(\mathbb{C}^{9})\oplus \Lambda^{3}(\mathbb{C}^{9})^{\ast}$ and denote by $\Lambda_{\mathbb{R}} := (\Lambda^{3}(\mathbb{C}^{9}))_{\mathbb{R}}$ the \emph{realification} of $\Lambda^{3}(\mathbb{C}^{9})$, which is a real vector space of real dimension $\dim_{\mathbb{R}}(\Lambda_{\mathbb{R}})= 2 \dim_{\mathbb{C}} \Lambda^{3}(\mathbb{C}^{9})$. We have
$\Lambda_{\mathbb{C}} =\Lambda_{\mathbb{R}}  \otimes_{\mathbb{R}} \mathbb{C}$ as a 
$\SU(9)/\mathbb{Z}_{3}$-representation, which means that $\Lambda_{\mathbb{C}}$ admits a $\SU(9)/\mathbb{Z}_{3}$ equivariant real structure. Hence, we conclude that the decomposition of $\mathfrak{e}^{\mathbb{C}}_{8}$ given in \eqref{eq:splittinge8C} admits an equivariant real structure and thus induces a decomposition of the real adjoint representation of $\E_{8}$ in real $\SU(9)/\mathbb{Z}_{3}$-representations given by
 \begin{equation} \label{eqn:su9}
 \mathfrak{e}_{8} = \mathfrak{su}(9)\oplus \Lambda_{\mathbb{R}}\, ,
 \end{equation} 
where $\Lambda_{\mathbb{R}}$ is the irreducible $168$-dimensional real representation of $\SU(9)/\mathbb{Z}_{3}$ 
described above. 
Note that $\Lambda_{\mathbb{R}} = (\Lambda^{3}(\mathbb{C}^{9}))_{\mathbb{R}}$ 
is a real representation of complex type, namely it admits an invariant complex structure.

\section{The orientation class}

Let $(M,\Omega)$ be a closed $8$-dimensional manifold endowed with a $\Spin(7)$-structure $\Omega$.
On the space of $2$-forms, there is an orthogonal decomposition $\Lambda^2=\Lambda^2_7\oplus \Lambda^2_{21}$
into irreducible $\Spin(7)$-representations, as explained in \cite{Joyce2007, Munoz2014}. The associated 
orthogonal projections will be denoted $\pi_7:\Lambda^2\to \Lambda^2_7$ and $\pi_{21}:\Lambda^2\to \Lambda^2_{21}$.

\begin{definition}
 A connection $A\in \cA$ is a $\Spin(7)$-instanton if $\pi_7(F_A)=0$.
\end{definition}

The instanton equation is elliptic modulo gauge transformations \cite{Carrion,DT1,MunozIII}. 
The moduli space of $\Spin(7)$-instantons is defined as
\begin{equation*}
 \frM := \left\{ A\in \cA\, \,\, |\,\, \pi_{7}(F_{A}) = 0\right\}/\bar{\cG}\subset \cB\, , 
\end{equation*}
and the moduli space of irreducible $\Spin(7)$-instantons is 
\begin{equation*}
\frM^{\ast} := \left\{ A\in \cA^{\ast}\, \,\, |\,\, \pi_{7}(F_{A}) = 0\right\}/\bar{\cG}\subset \cB^{\ast}\, .
\end{equation*}

In \cite{MunozIII} the authors initiated the study of $\frM^{\ast}$, proving that there exist suitable 
perturbations that achieve regularity of the moduli space, so that it is smooth and of the expected dimension.

Let us describe the tangent space at a point $A\in \frM^*$ from \cite{MunozIII}. 
Associated to the $\Spin(7)$-instanton $A$, there is a \emph{deformation complex}
 \begin{equation} \label{eq:complex}
 0\to \Omega^{0}(\ad(P)) \xrightarrow{d_{A}} \Omega^{1}(\ad(P)) 
 \xrightarrow{\pi_{7}\circ d_{A}} \Omega^{2}_{7}(\ad(P)) \to 0\, ,
 \end{equation}
where the spaces of sections have been completed on suitable Sobolev norms, so that they become Hilbert spaces. The complex (\ref{eq:complex}) is elliptic, so its cohomology groups $\HH^0_A, \HH^1_A,\HH^2_A$ are finite 
dimensional. 

The space $\HH^0_A$ is the Lie algebra of the stabilizer $\Gamma_A$ of the connection $A$. If $A$ is irreducible, then $\Gamma_A=Z(\G)$, the center of the Lie group, and $\HH^0_A=\text{Lie}(Z(\G))=\frz$, the center of the Lie 
algebra $\frg$. Therefore  
  \begin{equation}\label{eqn:O0}
  \Omega^{0}(\ad(P)) =\frz \oplus \Omega^{0}_\perp(\ad(P)),
  \end{equation}
where the space $\Omega^{0}_\perp(\ad(P))$ is defined as the orthogonal to $\frz$, and $d_A$ is injective on it.
The center $Z(\G)$ acts trivially on $\cA$, and we have a free action of $\bar\cG=\cG/Z(\G)$
on $\cA^*$. The Lie algebra of $\cG$ is $\Omega^0(\ad(P))$, and the 
Lie algebra of $\bar\cG$ is $\Omega^0_\perp(\ad(P))$. The 
tangent space to the orbit $\bar\cG\cdot A$ 
is the image of $d_A$. 
It is well-known (see for example \cite[Prop. 4.2.9]{DKbook} or \cite{MunozIII} for an explicit proof) that the space 
$\cB^*$ is locally modelled on $(\im d_A)^\perp =\ker d_A^{\ast}$.

Fix $A\in \frM^*$. For $a\in \Omega^1(\ad(P))$, we have $F_{A+a}=F_A+d_Aa+a\wedge a$. Therefore the equation $\pi_7(F_{A+a})=0$ is equivalent to $\pi_7(d_Aa+a\wedge a)=0$. This means that the moduli
space $\frM^*$ is locally modelled around $A$ on the solutions to 
 \begin{equation}\label{eqn:1}
 \cF(a)= \big( d_A^* a, \pi_7(d_Aa+ a\wedge a)\big)=0,
 \end{equation}
for $a\in \Omega^1(\ad (P))$, $||a||<\epsilon$, for some $\epsilon>0$. The linearization of (\ref{eqn:1}) is
the operator  
 \begin{align*}
  & Q_A: \Omega^1(\ad(P)) \too \Omega^0_\perp(\ad (P))\oplus \Omega^2_7(\ad (P)), \\
  &Q_A(a)= D_A\cF(a) =\big( d_A^* a,\pi_7(d_A a) \big)
 \end{align*}
The operator $Q_A$ is the rolled-up operator of the complex (\ref{eq:complex}). Therefore  
$$ 
 \ker Q_A=\HH^1_A, \qquad
 \coker Q_A=\HH^2_A.
 $$

We say that $A$ is regular if $\HH^2_A=0$, that is, if $Q_A$ is surjective. In that case, $0$ is a regular value
and the implicit function
theorem says that $\cF^{-1}(0)$ is a smooth manifold locally diffeomorphic to $\HH^1_A$. We say the the moduli
space $\frM^*$ is regular if all its points are regular. Suppose from now on that this is the case. Then the tangent
space is $T_A\frM^* \cong \HH^1_A$. The orientation line is then 
 \begin{equation}\label{eqn:oA}
 \fro_A : = \det (T_A\frM^{\ast}) = \det(\HH^1_A)\, .
 \end{equation}
This defines an \emph{orientation bundle} $\fro\to \frM^*$ whose fiber over $A\in \frM^*$ is $\fro_A$.

We remark that (\ref{eqn:oA}) coincides with the \emph{determinant line bundle}, which is defined as 
 $$
 \Det (Q_A):= (\det \ker Q_A)\ox (\det\coker Q_A)^*\, .
 $$
Let 
 \begin{equation}\label{eqn:2}
 Q_A': \Omega^1(\ad(P)) \too \Omega^0 (\ad (P))\oplus \Omega^2_7(\ad (P)) 
 \end{equation}
be defined as $Q_A$, but taking values on the full space $\Omega^0 (\ad(P))$.
Under (\ref{eqn:O0}), we have $\det(\coker Q_A)=\det(\coker Q_A') \ox (\det \frz)^* \cong 
\det(\coker Q_A^{\prime})$, since $\frz$ is a constant vector space. This means that  
 $$
 \fro_A=\Det (Q_A').
 $$

We can extend the bundle $\fro \to \frM^*$ to the whole of the configuration space $\cB^*$ as follows.
There is a description of the operator (\ref{eqn:2}) in terms of an 
appropriate Dirac operator. Let $(M,\Omega)$ be a closed $\Spin(7)$ manifold. In particular $M$ is spin, that is, it has a $\Spin(8)$-structure,
which is induced under the inclusion $\Spin(7)<\Spin(8)$. This can be rephrased as to saying that the 
frame bundle $\text{Fr}_{\SO(8)}(TM) \to M$ lifts under the double cover $\Spin(8)\to \SO(8)$, a condition that guarantees the existence of a spinor bundle, namely a bundle of irreducible real Clifford modules $S\to M$ over the bundle of Clifford algebras $\text{Cl}\,(TM)\to M$ of $M$. 
The spinor bundle $S$ decomposes as $S=S^+\oplus S^-$, where $S^{+}$ and $S^{-}$ are the rank $8$ bundles of positive and negative spinor bundles. Clifford multiplication by vectors $TM\subset \text{Cl}\,(TM)$ gives a map $c : TM \otimes S^\pm \to S^\mp$.
The Levi-Civita connection induces a connection on both $S^\pm$. 
Together with a choice of connection $A\in \cA$ on $P$, we get a connection $\nabla_A$ on the vector bundle 
$S^{\pm}\otimes \ad(P)$. Associated to this connection, we define the Dirac operator
 \begin{equation}\label{eqn:3}
 \frD_{A}\colon \Gamma(S^{-}\otimes \ad(P) ) \to \Gamma(S^{+}\otimes \ad(P) ),
 \end{equation}
via $\frD_A(s)= c(\nabla_A s)$, where $\nabla_A s\in \Gamma (T^*M\ox S^-\otimes \ad(P))$. 

The reduction of the $\Spin(8)$-structure to the $\Spin(7)$-structure is given by choosing a unit spinor $\eta \in \Gamma(S^+)$. Therefore $S^+ = \la \eta \ra \oplus H$, where $H=\la \eta\ra^\perp$ is a rank $7$ bundle. Clifford multiplication by $\eta$ gives an isomorphism $TM \to S^-$. Therefore, we have an isomorphism
\begin{equation}\label{eqn:4}
S^-\cong \Lambda^1.
\end{equation}

Using Clifford multiplication by two vectors, we get also a map $\Lambda^2 TM \to S^+$, which induces an isomorphism $\Lambda^2_7 TM \cong H$. Hence 
\begin{equation}\label{eqn:5}
S^+ \cong \RR \oplus \Lambda^2_7 .
\end{equation}

Under the isomorphisms (\ref{eqn:4}) and (\ref{eqn:5}), the map (\ref{eqn:2}) is rewritten as a map 
\begin{equation}\label{eqn:6}
\hat Q_{A}\colon \Gamma(S^{-}\otimes \ad(P)) \to \Gamma(S^{+}\otimes \ad(P)).
\end{equation}
The symbols of the maps (\ref{eqn:3}) and (\ref{eqn:6}) coincide \cite{MunozIII}. Therefore they are both Fredholm operators of the same index, and their 
{determinant line bundles} are canonically isomorphic. This implies that
 $$
 \fro_A=\Det (\frD_A).
 $$

The {determinant line bundle} $\Det(\frD)\to \cA^*$, whose fiber over $A\in \cA^*$ is $\Det(\frD_A)$, 
 is well-defined over the whole of $\cA^*$. It descends to a well-defined
line bundle ${\Det (\frD)}$ on $\cB^*$. It restricts to the orientation bundle under the inclusion $\frM^*\subset \cB^*$.

\begin{remark}\label{rem:or-red}
The  {determinant line bundle} $\Det (\frD)$ is also well-defined over $\cA$, and descends to a well defined 
line bundle ${\Det(\frD)}$ on $\cB_0$, since the action of $\cG_0$ on $\cA$ is free. 

Let $\pi:\cB_0 \to \cB$ be the quotient map and consider a reducible connection $[A]\in \cB$ such
that its stabilizer $\Gamma_A$ satisfies that $\Gamma_A/Z(\G)$ is connected. Then this group
acts trivially on the orientation line $\fro_A$, hence defining an orientation line $\fro_{[A]}$ over $[A]\in \cB$.
Therefore the orientation bundle can be extended to $\cB^*{}' \subset \cB$, the locus of
connections with $\Gamma_A/Z(\G)$ connected. 
\end{remark}

\begin{definition}
 The \emph{orientation class} is the class $W=w_1(\Det (\frD)) \in H^1(\cB^*,\ZZ_2)$, where $w_1$ denotes the
first Stiefel-Whitney class.
\end{definition}

The orientation class $W$ controls the orientation of $\frM^{\ast}$. The map
  \begin{equation}\label{eqn:orientab}
 \pi_1(\frM^*) \to \pi_1(\cB^*) \to \ZZ_2
 \end{equation}
given by $\gamma \mapsto \la W,\gamma\ra$ determines the orientation around the
loop $\gamma\in \pi_1(\frM^*)$. 

\medskip

Now we move to the more general case in which the moduli space $\frM^*$ is not regular. 
In \cite{MunozIII} we defined different sets of perturbations which satisfy some nice properties. The $\Spin(7)$-instanton equation is of the form
$E:\cA \to \Omega^2_7(\ad P)$, and it is $\cG$-invariant. We consider a Banach space
of perturbations $\Pi$ such that: (1) every $\varpi\in \Pi$ corresponds to an equation
$E_\varpi:\cA \to \Omega^2_7(\ad P)$, which is $\cG$-invariant, (2) the equation
for some $\varpi_0$ is the original equation $E=E_{\varpi_0}$, (3) the map 
$\cE: \cA\x \Pi \to \Omega^2_7(\ad P)$, $\cE(A,\varpi)=E_\varpi(A)$, is smooth, 
(4) for a small neighbourhood
of $\varpi_0$ in $\Pi$, all maps $\cE_\varpi$ have elliptic linearizations (coupled with 
the action of the gauge group, as above). 

A perturbation $\varpi\in \Pi$ gives rise to a moduli space $\frM_\varpi^*$. If all solutions
are regular, then the moduli space $\frM_\varpi^*$ is a smooth finite dimensional moduli
space of the expected dimension. In this case we say that the perturbed moduli space
$\frM_\varpi^*$ is regular. As in the previous situation, the determinant line bundle $\Det(\frD)$ induces 
the orientation bundle on $\frM_\varpi^*\subset \cB^*$.

\begin{remark}
We can also define the concept of orientability for non-regular moduli spaces $\frM^*$, that is, when
$\frM^*$ is non-smooth or when $\HH^2_A\neq 0$, for $A\in \frM^*$. In this case we say that
the orientation bundle is the restriction of $\fro\to \cB^*$ to $\frM^*\subset \cB^*$, and we say
that $\frM^*$ is orientable if $\fro|_{\frM^*}$ is a trivial real line bundle. With this notion,
the same results of this paper hold for non-regular moduli spaces of $\Spin(7)$-instantons.
\end{remark}

\section{The fundamental group of the configuration space for $\G=\E_8$} \label{sec:or-class} 

In this section we shall compute the fundamental group $\pi_1(\cB^*)$ for the case $\G=\E_8$. 
Since $\pi_{3}(\E_{8}) \cong \mathbb{Z}$ and $\pi_{i}(\E_{8}) = 0$ for $0\leq i\leq 14$, $i\neq 3$, 
by Theorem \ref{thm:piE8}, we have a homotopy equivalence up to degree $14$,
 \begin{equation*}
 \E_{8} \simeq_{14} K(3,\mathbb{Z})\, ,
 \end{equation*}
where $X \simeq_{k} Y$ means that $X,Y$ have homotopy equivalent $k$-skeleta, $k\geq 1$.
For the classifying space, we have $B\E_{8}\simeq_{15} BK(3,\ZZ)=K(4,\mathbb{Z})$. 

The isomorphism classes of smooth $\E_{8}$-bundles over $M$ are in one to one correspondence with 
homotopy classes of maps from $M$ to $B\E_8$. These are classified by  
 $H^{1}(M,C^{\infty}(\E_{8})) \cong [M,B\E_8]$. 
As $M$ is an $8$-dimensional CW-complex, we can
substitute $B\E_8$ by $K(4,\ZZ)$, that is, $[M,B\E_8] =[M,K(4,\ZZ)]$. Finally, we have the well-known isomorphism
 $$
 [M,K(4,\mathbb{Z})] \cong H^{4}(M,\mathbb{Z})\, .
 $$

\begin{prop}
\label{prop:E8orientability}
Let $M$ be a closed oriented and spin $8$-dimensional manifold, and let $\G=\E_8$. Then 
$\pi_1(\cB^*)=H^3(M,\ZZ)$.
\end{prop}

We will provide three independent proofs of Proposition \ref{prop:E8orientability}.
Our first proof follows the arguments of  \cite{CaoConan,CaoConanII} for the case of $8$-manifolds
with $\SU(4)$-structures. 

\begin{proof}
Recall that by (\ref{eqn:hom-B-cB}) and Propostion \ref{prop:homotopyB0}, we have that 
$\cB^*\simeq B\cG \simeq \Map^P(M,B\E_8)$, using that $\E_8$ is simply-connected and $\cG=\bar\cG$.
We apply the Federer spectral sequence to $M$ and $B\E_{8}$. It has second page given by 
 \begin{equation} \label{eqn:E2}
 E^{p,q}_{2}= H^{p}(M,\pi_{p+q}(B\E_{8})),
 \end{equation} 
and abuts to $\pi_{q}(\Map^{P}(M,B\E_{8}))$. By Theorem \ref{thm:piE8}, the only non-zero homotopy group
for $1\leq l \leq 14$ is $\pi_4(B\E_8)=\pi_3(\E_8)=\ZZ$. Therefore only the case $p+q=4$ appears in (\ref{eqn:E2}).
This implies that there are no differentials on $E^{p,q}_2$, and hence we have an isomorphism
 $$
 \pi_q(\Map^{P}(M,B\E_{8})) \cong H^{4-q}(M,\ZZ),
 $$
since it must be $p+q=4$. Particularizing to $q=1$, we have
 $$
 \pi_1(\cB^*) \cong \pi_1(\Map^{P}(M,B\E_{8})) \cong H^{3}(M,\ZZ).
 $$
\end{proof}
 
\begin{cor}
\label{cor:orientabilityclass}	
For $\G=\E_8$, and a regular moduli space $\frM^*$, 
if $\Hom(H^{3}(M,\mathbb{Z}),\mathbb{Z}_{2}) = 0$ then $\frM^*$ is orientable.
\end{cor}

\begin{proof}
The orientability of $\frM^{\ast}$ is controlled by the homomorphism 
(\ref{eqn:orientab}). This factors as 
$\pi_{1}(\frM^{\ast}) \to\pi_{1}(\cB^{\ast}) \to \mathbb{Z}_{2}$, so
by a morphism $\pi_1(\cB^*)\cong H^3(M,\ZZ) \to \ZZ_2$. Therefore if 
$\Hom(H^{3}(M,\mathbb{Z}),\mathbb{Z}_{2}) = 0$, then the homomorphism
(\ref{eqn:orientab}) is zero, and $\frM^*$ is orientable.
\end{proof}

We give now the second proof of Proposition \ref{prop:E8orientability}, which does not rely on the Federer spectral sequence. 

\begin{proof}
By (\ref{eqn:hom-B-cB}) and (\ref{eqn:pi1}), we have that $\pi_1(\cB^*)\cong \pi_1(\cB_0)$, since $\G=\E_8$ is
simply-connected. By Proposition \ref{prop:homotopyB0}, we have that $\cB_0\simeq 
B\cG_{0} \simeq\Map^{P}_{\ast}(M,B\E_8)$. So
we have  to compute 
 $$
 \pi_1(\cB^*) \cong \pi_1(\cB_0) \cong \pi_1(\Map^{P}_{\ast}(M,B\E_8)). 
 $$
As $M$ is an $8$-dimensional CW-complex, we can
substitute $B\E_8$ by $K(4,\ZZ)$. Hence 
 $$
 \pi_1(\cB^*) \cong \pi_1(\Map_{\ast}^P(M,K(4,\ZZ))).
 $$

Note that for any space $X$, the fibration $\Map_*(X,K(4,\ZZ)) \to \Map(X,K(4,\ZZ)) \to K(4,\ZZ)$
implies that $\Map_*(X,K(4,\ZZ)) \simeq_3 \Map(X,K(4,\ZZ))$, so 
 \begin{equation}\label{eqn:pik}
 \pi_k (\Map_*(X,K(4,\ZZ)))=\pi_k(\Map(X,K(4,\ZZ))), \quad \text{for }k\leq 2.
\end{equation}
Therefore
 $$
 \pi_1(\cB^*) \cong \pi_1(\Map^P(M,K(4,\ZZ))) =\pi_0(\Omega \Map^P (M,K(4,\ZZ))).
 $$

We have a fibration 
 $$
 \Omega\Map (M,K(4,\ZZ)) \to \Map (M\x S^1 , K(4,\ZZ)) \to \Map (M,K(4,\ZZ)) .
 $$
If we restrict to the connected component that induces the bundle $P$, we have a fibration
 $$
 \Omega\Map^P (M,K(4,\ZZ)) \to \Map^P(M\x S^1 , K(4,\ZZ)) \to \Map^P(M,K(4,\ZZ)) 
 $$

The map $\pi_1(\Map(M\x S^1 , K(4,\ZZ))) \to \pi_1(\Map(M,K(4,\ZZ)))$
is surjective, since it admits an splitting via $\pi:M\x S^1\to M$, $f\mapsto f\circ \pi$. Hence we have a short exact sequence
 \begin{equation}\label{eqn:more}
 \pi_0( \Omega\Map^P (M,K(4,\ZZ))) \to \pi_0(\Map^P(M\x S^1 , K(4,\ZZ))) \to \pi_0(\Map^P(M,K(4,\ZZ))) .
 \end{equation}

There are natural isomorphisms
 \begin{align*}
  \pi_0(\Map_{*}(M, K(4,\mathbb{Z}))) &=[M,K(4,\ZZ)]_* \cong H^4(M,\ZZ), \\
  \pi_0(\Map_{*}(M\x S^1, K(4,\mathbb{Z}))) &=[M\x S^1,K(4,\ZZ)]_* \cong H^4(M\x S^1,\ZZ) \\
 &= H^4(M,\ZZ) \oplus \big(H^{3}(M,\mathbb{Z})\otimes H^{1}(S^{1},\mathbb{Z})) \\ 
 &\cong H^4(M,\ZZ) \oplus H^{3}(M,\mathbb{Z}).
 \end{align*}
By (\ref{eqn:pik}), we rewrite (\ref{eqn:more}) as
 $$
 \pi_0( \Omega\Map^P (M,K(4,\ZZ))) \to  H^4(M,\ZZ) \oplus H^{3}(M,\mathbb{Z}) \to H^4(M,\ZZ).
 $$
Therefore
 $$
 \pi_1(\cB_0)\cong \pi_0( \Omega\Map^P (M,K(4,\ZZ)) ) \cong H^3(M,\ZZ),
 $$
as required.
\end{proof}

We will give a third proof of Proposition \ref{prop:E8orientability} inspired by an argument in \cite[Section 5]{DKbook}. 
It is based on considering a cellular decomposition of $M$.

\begin{proof}
By Morse theory, we have a self-index Morse function $f:M\to \RR$ such that a critical point $x$ of index $k$ has $f(x)=k$.
Then $M_k'= f^{-1}((-\infty, k+\frac12])$ is a smooth manifold with boundary which is homotopy equivalent to the
$k$-skeleton $M_k$ of $M$. Then we have cofibrations 
 \begin{equation} \label{eq:celco}
 M_{i-1}\hookrightarrow M_{i} \to \bigvee^{n_{i}}_{1} S^{i}\, ,
 \end{equation}
where $n_{i}$ is the number of $i$-cells of the cellular decomposition of $M$. We need the following.

\begin{lemma} \label{lemma:E8Sn}
Let $P\to S^{n}$ be a $\G = \E_{8}$ principal bundle over the $n$-sphere $S^{n}$. Then if $n\leq 8$ and $n\neq 3$ we have $\pi_{1}(\cB^*) \cong \pi_{1}(\cB_{0}) = 1$, whereas if $n=3$ then $\pi_{1}(\cB^*)\cong \pi_{1}(\cB_{0}) \cong\mathbb{Z}$.

Likewise, $\pi_k(\cB_0)=\ZZ$ for $S^n$ with $n=4-k$, and $0$ otherwise.
\end{lemma}

\begin{proof}
Since $S^{n}$ is a suspension, by Remark \ref{remark:COH}, we have that Proposition \ref{prop:thm:pahtcomp} applies. Thus
$\cG_{0} \cong \Map_{\ast}(S^{n}, \G) \cong \Omega^{n}\G$. Hence, using Lemma \ref{lem:G-sc-noZ}, we have
 $$
 \pi_1(\cB^*)= \pi_1(\cB_0)=\pi_1(B\cG_0)=\pi_0(\cG_0) =\pi_0(\Omega^n\G)=\pi_n(\G),
 $$
and the result follows by using the homotopy groups of $\G=\E_8$ given in Theorem \ref{thm:piE8}.

The second statement is analogous.
\end{proof}

Now we apply now the exact contravariant functor $\Map_{\ast}(- , B\G)$ to the cofibration 
\eqref{eq:celco}. We obtain the following fibration
 \begin{equation*}
 \prod^{n_{i}}_{1}\, \Map_{\ast}(S^{i},B\G)\to \Map_{\ast}(M_{i},B\G) \to \Map_{\ast} (M_{i-1},B\G)\, .
 \end{equation*} 

Note that $\pi_0(\Map_*(M_i,B\G))=[M_i,B\G]=H^4(M_i,\ZZ)=0$ for $i\leq 3$. Considering the connected component corresponding to the map in $\Map_*(M_i,B\G)$ defining the bundle $P$,
we have a fibration 
  \begin{equation*}
  \prod^{n_{i}}_{1}\, \Map^{P|_{S^{i}}}_{\ast}(S^{i},B\G)\to \Map^{P|_{M_{i}}}_{\ast}(M_{i},B\G) \to 
  \Map^{P|_{M_{i-1}}}_{\ast} (M_{i-1},B\G)\, ,
  \end{equation*} 
which, by means of Proposition \ref{prop:homotopyB0}, implies the following fibration
  \begin{equation} \label{eq:iterativoB0}
  \prod^{n_{i}}_{1}\, \cB_{0}(P|_{S^{i}}) \to \cB_{0}(P|_{M_{i}}) \to \cB_{0}(P|_{M_{i-1}})\, ,
  \end{equation}
where $\cB_{0}(P|_{S^{i}})$ denotes the space of connections modulo the framed gauge group on the bundle induced by $P$ on $S^{i}$, and likewise for $\cB_{0}(P|_{M_{i}})$ and $\cB_{0}(P|_{M_{i-1}})$ on $M_{i}$ and $M_{i-1}$, respectively. 

From (\ref{eq:iterativoB0}) and Lemma \ref{lemma:E8Sn}, we get inductively that 
 $$
 \pi_1(\cB_0(P|_{M_2}))=1.
 $$
Now we get an exact sequence
 $$
 \pi_2(\cB_0(P|_{M_2})) \to  \prod^{n_{3}}_{1}\,\pi_{1}(\cB_{0}(P|_{S^{3}})) 
 \to \pi_{1}(\cB_{0}(P|_{M_{3}})) \to 1.
 $$
As $\pi_2(\cB_0(P|_{M_1}))=1$, we have a surjection 
   $$
  \prod^{n_{2}}_{1}\,\pi_{2}(\cB_{0}(P|_{S^{2}}))  \to \pi_{2}(\cB_{0}(P|_{M_{2}})) \to 1,
 $$
and composing, we get an exact sequence 
 $$
  \prod^{n_{2}}_{1}\,\pi_{2}(\cB_{0}(P|_{S^{2}})) \to  \prod^{n_{3}}_{1}\,\pi_{1}(\cB_{0}(P|_{S^{3}})) 
 \to \pi_{1}(\cB_{0}(P|_{M_{3}})) \to 1.
 $$

There is a natural identification $C^2_{cel}(M_2)=\prod_1^{n_2}\ZZ= \prod^{n_{2}}_{1}\,\pi_{2}(\cB_{0}(P|_{S^{2}}))$,
where $C^k_{cel}(M)$ is the chain complex of cellular chains. So this gives an exact sequence
 $$
 C^2_{cel}(M) \to C^3_{cel}(M) \to \pi_1(\cB_{0}(P|_{M_{3}})) \to 1.
 $$
The first map is identified with the coboundary map $\partial_{cel}^2$. So $\pi_1(\cB_{0}(P|_{M_{3}})) =
C^3_{cel}(M)/\im \partial_{cel}^2$

Now use again  (\ref{eq:iterativoB0})  to get an exact sequence
 $$
 1 \to  \pi_{1}(\cB_{0}(P|_{M_{4}})) \to  
  \pi_1(\cB_0(P|_{M_3})) \to  \prod^{n_{4}}_{1}\,\pi_{0}(\cB_{0}(P|_{S^{4}})) ,
 $$
which is rewritten as
 $$
 1 \to  \pi_{1}(\cB_{0}(P|_{M_{4}}))  \to \frac{C^3_{cel}(M)}{\im \partial_{cel}^2} \to C^4_{cel}(M).
 $$
The last map is identified with $\partial_{cel}^3$. Hence
 $$
\pi_{1}(\cB_{0}(P|_{M_{4}}))  \cong \frac{\ker \partial_{cel}^3}{\im \partial^2_{cel}} \cong H^3(M,\ZZ).
 $$

Finally, we inductively get that $\pi_1(\cB_0(P|_{M_k}))=\pi_1(\cB_0(P|_{M_4}))$, for $k>4$,
and hence we conclude.
\end{proof}

\section{Orientability of the moduli space for $\G=\SU(r)$}

In the previous sections, we have discussed the orientability of the moduli space of $\Spin(7)$-instantons 
for a principal bundle $P\to X$ with gauge group $\G=\E_8$. 
The choice of this group is due to the fact that it has trivial center and all its homotopy groups but $\pi_3(\E_8)$ are trivial in the range that we need.
Now it is our task to translate the orientability property for the case of $\G=\E_8$ to principal bundles with the more usual Lie groups
$\G=\SU(r)$. 

The first step is to move from $\E_8$ to $\SU(9)/\ZZ_3$, using the inclusion $\SU(9)/\ZZ_3\subset \E_8$ described in 
Section \ref{sec:E8}. We shall use Donaldson's stabilization argument in \cite{DTOP,DOR} but applied to $\E_8$ instead of $\SU(n)$.

\begin{prop}
\label{prop:thm:orientability1}
Let $M$ be a closed oriented and $\Spin(7)$-manifold with $\Hom(H^{3}(M,\mathbb{Z}),\mathbb{Z}_{2}) = 0$. 
Then the moduli space $\frM^{\ast}\subset \cB^{\ast}$ of $\Spin(7)$-instantons for the gauge group $\G = \SU(9)/\ZZ_3$ is orientable
(assuming it is regular).
\end{prop}

\begin{proof}
Consider a principal $\SU(9)/\mathbb{Z}_{3}$-bundle $Q$. Associated to the embedding $i\colon \SU(9)/\mathbb{Z}_{3}\hookrightarrow \E_{8}$ through left-multiplication we construct a principal $\E_{8}$-bundle
 \begin{equation*} 
 P = Q\times_{i} \E_{8}\, ,
 \end{equation*}
to which we associate the following vector bundle
 \begin{equation*}
  \ad(P) = P\times_{\ad} \mathfrak{e}_{8}\, .
 \end{equation*}
where $\ad$ denotes the adjoint representation. 
Since $P$ is associated to $Q$ through the embedding $i\colon\SU(9)/\mathbb{Z}_{3}\hookrightarrow \E_{8}$, we can write $\ad(P)$ as 
 \begin{equation*}
 \ad(P) = Q\times_{\rho} (\mathfrak{su}(9) \oplus \Lambda_{\mathbb{R}})\, ,
 \end{equation*}
where $\rho$ denotes the decomposition of the real adjoint representation of $\E_{8}$ in $\SU(9)/\mathbb{Z}_{3}$-representations as prescribed by equation \eqref{eqn:su9}. Hence, we obtain
 \begin{equation*} 
 \ad(P) \cong \ad(Q) \oplus E_{\Lambda_{\mathbb{R}}}\, , 
 \end{equation*}
where $\ad(Q)$ is the real adjoint bundle of $Q$ and $E_{\Lambda_{\mathbb{R}}}$ is a rank $168$ real vector bundle admitting complex multiplication and hence canonically orientable. 

There is a natural map $s\colon \cB^{\ast}(Q)\rightarrow \cB (P)$ from the space of irreducible connections on 
the $\SU(9)/\mathbb{Z}_{3}$-bundle $Q$ modulo gauge transformations,
 to the space of connections on the $\E_8$-bundle $P$ modulo gauge transformations. Note that 
the image of this map sits in the locus of reducible connections. However, Remark \ref{rem:or-red} applies and
the orientation bundle can be extended over the image of $s$.

Under the assumptions made in the statement, 
Proposition \ref{prop:E8orientability} implies now that the determinant line bundle 
$\Det(\frD)$ is trivial when restricted to closed loops in $\cB^{\ast}(P)$, and hence also over the image of $s$.
The pull-back of the determinant line bundle $\Det(\frD)$ by $s$ can be written as
 \begin{equation*}
  s^{\ast} \Det (\frD,\ad(P)) = \Det(\frD, \ad( Q)   \oplus E_{\Lambda_{\mathbb{R}}})\, .
 \end{equation*}
where the right hand side denotes the determinant line bundle of the Dirac operator over $\cB^{\ast}(Q)$ coupled to the adjoint bundle $\ad(P)$ 
decomposed as an associated bundle of ${\SU(3)/\mathbb{Z}_{3}}$ as described in Section \ref{sec:E8}. Hence we obtain
 \begin{equation*}
  s^{\ast} \Det(\frD,\ad(P)) = \Det (\frD,  \ad( Q))\otimes\Det(\frD, E_{\Lambda_{\mathbb{R}}})\, .
 \end{equation*}

As shown in Section \ref{sec:E8}, $E_{\Lambda_{\mathbb{R}}}$ admits a canonical orientation 
induced by a complex structure, we conclude that $\Det(\frD, E_{\Lambda_{\mathbb{R}}})$ is canonically trivial. 
This proves that 
$\Det(\frD, \ad(Q))$ is trivial and hence shows that the moduli space of irreducible 
$\SU(9)/\mathbb{Z}_{3}$-connections is orientable. 
\end{proof}

We are now ready to our last step.

\begin{thm}
Let $M$ be a closed oriented and $\Spin(7)$-manifold with $\Hom(H^{3}(M,\mathbb{Z}),\mathbb{Z}_{2}) = 0$. Then the smooth moduli space $\frM^{\ast}\subset \cB^{\ast}$ of $\Spin(7)$-instantons for the gauge group $\G = \SU(r)$, $r\geq 2$, is orientable (assuming it is regular).
\end{thm}

\begin{proof}
We start by moving from the group $\SU(9)/\ZZ_3$ to $\SU(9)$. Consider a principal $\SU(9)$-bundle $P_{\SU(9)}$.
We take the bundle $P_{\SU(9)/\mathbb{Z}_{3}}$ associated to $P_{\SU(9)}$ through the 
projection $p\colon \SU(9)\to \SU(9)/\mathbb{Z}_{3}$.
Certainly, there exists a $\mathbb{Z}_{3}$-covering map $\bar{p}\colon P_{\SU(9)}\to P_{\SU(9)/\mathbb{Z}_{3}}$, 
which in particular implies $P_{\SU(9)/\mathbb{Z}_{3}}\cong P_{\SU(9)}/\mathbb{Z}_{3}$, where the $\mathbb{Z}_{3}$-action is induced by the
$\SU(9)$-action on $P_{\SU(9)}$ through the inclusion $\mathbb{Z}_{3}\hookrightarrow \SU(9)$. 

{}From the inclusion $\mathbb{Z}_{3}\subset Z(\SU(9))$, it follows that $\mathbb{Z}_{3}$ acts trivially on the space of connections and hence $\bar{p}$ induces a surjective map 
$q \colon \cB^{\ast}(P_{\SU(9)})\to \cB^{\ast}(P_{\SU(9)/\mathbb{Z}_{3}})$ at the level of gauge equivalence classes of connections. The preimage by $q$ of any given point in $\cB^{\ast}(P_{\SU(9)/\mathbb{Z}_{3}})$ is a torsor over $\Hom(\pi_{1}(M),\mathbb{Z}_{3})\cong H^{1}(M,\mathbb{Z}_{3})$. 
Therefore there is an injective map
 $$
 q_*:\pi_{1}(\cB^*(P_{\SU(9)})) \rightarrow \pi_{1}(\cB^*(P_{\SU(9)/\mathbb{Z}_{3}})).
 $$

Under the assumptions made in the statement, 
Proposition \ref{prop:thm:orientability1} implies that the determinant bundle of
$\Det(\frD, \ad(P_{\SU(9)/\ZZ_3}))$ is trivial when restricted to closed loops in $\cB^{\ast}(P_{\SU(9)/\ZZ_3})$. 
Now, the adjoint bundles $\ad(P_{\SU(9)/\ZZ_3})\cong\ad(P_{\SU(9)})$ are isomorphic. Hence
$q_*(\Det(\frD, \ad(P_{\SU(9)})))$ is also trivial. Therefore the orientation class $W(P_{\SU(9)})=0$ vanishes.

The next step is to use Donaldson's stabilization argument to move from $\SU(9)$ to $\SU(r)$, for $ r\leq 9$.  We do this step-wise. 
Let $P_{\SU(l)}$ be a principal $\SU(l)$-bundle, with $l\leq 8$, and consider the inclusion $i\colon \SU(l)\hookrightarrow \SU(l+1)$. There is an induced map 
\begin{equation*}
i_* : \cB^*(P_{\SU(l)})\to \cB(P_{\SU(l+1)}),
\end{equation*}
which sends a connection $A$ on an associated complex rank $l$ bundle $E_{\SU(l)}$ to the connection on $E_{\SU(l+1)}=E_{\SU(l)}\oplus \underline{\CC}$ 
which is $A$ on the first summand and the trivial connection on the second summand. Note that the image of 
$i_*$ lies in the locus of connections with connected stabilizer, so Remark \ref{rem:or-red} can be applied.

By induction hypothesis, $W=0$ on $\cB^*(P_{\SU(l+1)})$ and by 
Remark \ref{rem:or-red} the same holds on the locus $\cB^*{}'(P_{\SU(l+1)})$. 
This implies that the {determinant line bundle} is trivial on loops in the image
$i_*(\cB^*(P_{\SU(l)}))$. Finally, there is an isomorphism 
 $$
 \ad(P_{\SU(l+1)}) \cong \ad(P_{\SU(l)}) \oplus \underline\CC^{l+1}\, ,
 $$
where the second summand has a natural complex structure. Therefore the argument of Proposition \ref{prop:thm:orientability1} can be applied here to prove that $W(\cB^*(P_{\SU(l)}))=0$, for $l\leq 8$. This implies that the moduli space is orientable for $r\leq 9$.

To finish, we prove that $W(\cB^*(P_{\SU(r)}))=0$ also for $r\geq 10$. Let us pick a loop $[\gamma]\in \pi_{1}(\cB^{\ast}(P_{\SU(r)}))$. 
Using the bijection
   \begin{equation*}
   \pi_{1}(\cB^{\ast}(P_{\SU(r)})) = \pi_{0}(\bar{\cG}(P_{\SU(r)}))\, ,
   \end{equation*}
the homotopy class $[\gamma]$ defines a unique element $[\bar{\phi}] \in \pi_{0}(\bar{\cG}(P_{\SU(r)}))$.
As the reduced gauge group is $\bar{\cG}(P_{\SU(r)})= \cG(P_{\SU(r)})/ \mathbb{Z}_{r}$, we can take an element $\phi\in \cG(P_{\SU(r)})$ mapping to $[\bar{\phi}]$ through the obvious projection. There are
$r$ choices for the different preimages of $[\bar\phi]$ in $\pi_0(\cG(P_{\SU(r)}))$.
 
We construct the principal bundle $P_\phi \to M\x S^1$, by doing the mapping torus using the map $\phi$ acting on 
the principal bundle $P_{\SU(r)}$. The class $[\phi]\in \pi_0(\cG(P_{\SU(r)}))$ defines the principal bundle 
$P_{\phi} \to M\times S^{1}$ uniquely up to isomorphism. This bundle is in turn uniquely determined by the homotopy class
$[f]\in [M\x S^1, B\SU(r)]$ of the classifying map $f:M\x S^1 \to B\SU(r)$. The restriction to $M$ is 
$P_\phi|_{M} =P_{\SU(r)}$, which corresponds to a fixed homotopy class $[f_{P_{\SU(r)}}]\in [M, B\SU(r)]$. Therefore
there is a natural bijection 
    \begin{equation*}
    \pi_{0}(\cG(P_{\SU(r)})) \cong \left\{ [f] \in [M\times S^{1}, B\SU(r)]\,\, |\,\, [f|_{M}] = [f_{P_{\SU(r)}}] \in [M,B\SU(r)] \right\}.
    \end{equation*}
Now note that $\SU(r)\simeq_{18} \SU(9)$, that is, they are homotopy equivalent up to the $18$-skeleton. Therefore
$B\SU(r)\simeq_{19}  B\SU(9)$, whence we obtain the following isomorphisms
    \begin{equation*}
   [M\times S^{1}, B\SU(9)] \cong [M\times S^{1}, B\SU(r)] \ \text{ and } \ 
   [M, B\SU(9)] \cong [M, B\SU(r)].
   \end{equation*}

Thus there exist principal $\SU(9)$-bundles $P_{\SU(9)}' \to M$ and  $P^{\prime}_{\phi'} \to M\times S^{1}$ such that
   \begin{equation*} 
   P_{\phi} = P^{\prime}_{\phi'} \times_{\SU(9)} \SU(r)\ \text{ and } \  P_{\SU(r)} = P_{\SU(9)}' \x_{\SU(9)} \SU(r) .
   \end{equation*}
This means that the fiber bundle $P_{\SU(r)}$ admits a topological reduction to a principal $\SU(9)$-bundle 
$P^{\prime}_{\SU(9)}$ over $M$, and analogously for $P_\phi$ and $P'_{\phi'}$ over $M\x S^1$. Moreover, the
restriction of $P^{\prime}_{\phi'}$ to $M$ is clearly isomorphic to $P'_{\SU(9)}$. 
The element $\phi'\in \cG(P'_{\SU(9)})$ defines a loop $[\gamma^{\prime}] \in \pi_{1}(\cB(P^{\prime}_{\SU(9)}))$
that induces $P^{\prime}_{\phi'}$ in the same way as $[\gamma]$ gives rise to $P_{\phi}$. 
As $P_{\SU(r)}$ reduces to $P^{\prime}_{\SU(9)}$, we have a canonical inclusion
   \begin{equation*}
   i\colon \cB^{\ast}(P^{\prime}_{\SU(9)}) \to \cB^*{}'(P_{\SU(r)})\, , 
   \end{equation*}
that we can use to push-forward $[\gamma^{\prime}]\in \pi_{1}(\cB^{\ast}(P^{\prime}_{\SU(9)}))$ to $i_{\ast}[\gamma^{\prime}] \in \pi_{1}(\cB(P_{\SU(r)}))$. 
Clearly, $[\gamma] = i_{\ast}[\gamma^{\prime}]$.

By the first part of the proof, we know that the determinant line bundle is trivial over $[\gamma^{\prime}]$. Here we apply Remark \ref{rem:or-red};
alternatively, we take the reducible connection on $\ad(P'_{\SU(9)}) \oplus \ad(\underline\CC^{r-9}) \subset \ad(P_{\SU(r)})$, determined by the loop 
$\gamma'$ on the bundle $P'_{\SU(9)}$ and the trivial connection on the second summand, and we perturb it to make it irreducible.
Therefore the determinant line bundle is trivial when restricted to $i_{\ast}[\gamma^{\prime}]$. As $[\gamma]=i_*[\gamma']$, 
the determinant line bundle is also trivial over the initial loop $[\gamma]$ and we conclude.
\end{proof}

\begin{remark} \label{rem:6.3}
The quaternionic projective space $\mathbb{HP}^{2}$ is an example of an $8$-dimensional 
$\Spin(7)$-manifold satisfying $H^{3}(M,\mathbb{Z}) = 0$ and in particular $\Hom(H^{3}(M,\mathbb{Z}),\mathbb{Z}_{2}) = 0$. Another example of manifold admitting a (generically non-integrable) $\Spin(7)$-structure and satisfying $H^{3}(M,\mathbb{Z}) = 0$ is given by the 
$8$-dimensional complex Grassmanian $\mathrm{Gr}_{2}(\mathbb{C}^{4})$, which is a particular case of \emph{Wolf space}. 
\end{remark}


\phantomsection
\bibliographystyle{JHEP}

\begin{thebibliography}{99}

\bibitem{Adams}
J. Adams,  \emph{Lectures on Exceptional Lie groups}, Chicago Lectures in Mathematics Series, 1996.

\bibitem{AtiyahBott} 
M. Atiyah and R. Bott, \emph{The Yang-Mills Equations over Riemann Surfaces}, Phil. Trans. Roy. Soc. London A 308 (1983) 523-615.

\bibitem{AtiyahHitchin} 
M. Atiyah, N. Hitchin and I. Singer, \emph{Self-duality in four-dimensional Riemannian geometry},
 Proc. Roy. Soc. London A 362  (1978) 425-461.

\bibitem{BBS} 
K.  Becker, M. Becker and J. Schwarz, \emph{String Theory and M-Theory: A Modern Introduction}, Cambridge University Press, 2007.

\bibitem{BottSamelson} 
R. Bott and H. Samelson, \emph{Application of the theory of Morse to symmetric spaces}, Amer. Jour. Math. 80 (1958) 964-1029.

\bibitem{ConjugacyClasses}
P. Booth, P. Heath, C. Morgan and R. Piccinini, 
\emph{H-Spaces of Self-Equivalences of Fibrations and Bundles}, Proc. London Math. Soc. 49 (1984) 111-127. 



\bibitem{CaoConan}
Y. Cao and N. Leung, \emph{Donaldson-Thomas theory for Calabi-Yau 4-folds,} arXiv:1407.7659

\bibitem{CaoConanII}
Y. Cao and N. Leung, \emph{Orientability for gauge theories on Calabi-Yau manifolds,} arXiv:1502.01141



\bibitem{CDFN}
E. Corrigan, C. Devchand, D. Fairlie and J. Nuyts, \emph{First order equations for gauge fields in spaces of dimension greater than four}, 
Nuclear Phys. B, 214 (1983)  452-464.

\bibitem{DTOP}
S. Donaldson, \emph{An application of gauge theory to four-dimensional topology}, J. Diff. Geom. 18 (1983), 279-315.

\bibitem{Dbook}
S. Donaldson, \emph{Floer Homology Groups in Yang-Mills Theory}, Cambridge University Press, 2002.

\bibitem{DOR}
S. Donaldson, \emph{The orientation of Yang Mills moduli spaces and 4-manifold topology}, J. Diff. Geom. 26 (1987) 397-428.

\bibitem{DKbook}
S. Donaldson and P. Kronheimer, \emph{The Geometry of Four-manifolds}, Oxford Mathematical Monographs, 1997.

\bibitem{DT1}
S. Donaldson and R. Thomas, \emph{Gauge theory in higher dimensions}, in 
``{The Geometric Universe: Science, Geometry, and the Work of Roger Penrose}'', Oxford University Press, 1998.




\bibitem{Fernandez-Gray}
M. Fern\'andez and A. Gray, \emph{Riemannian manifolds with structure group $G_2$}, Annali Mat. Pura Appl. 132 (1982) 19-45.




\bibitem{FreedUhlenbeck}
D. Freed and K. Uhlenbeck, \emph{Instantons and four-manifolds}, Springer-Verlag, 1991.












\bibitem{Joyce2007}
D. Joyce, \emph{Compact manifolds with special holonomy}, Oxford Mathematical Monographs, 2000.





\bibitem{Lewis}
C. Lewis, \emph{$\Spin(7)$ instantons}, Oxford University D.Phil thesis, 1998.


\bibitem{Munoz2014}
V. Mu\~noz, \emph{$\Spin(7)$-instantons, stable bundles and the Bogomolov inequality for complex $4$-tori}, 
J.  Math. Pures Appl. 102 (2014) 124-152.


\bibitem{MunozIII}
 V.~Mu\~noz and C.S.~Shahbazi, \emph{Construction of the moduli space of $\Spin (7)$-instantons,}
  arXiv:1611.04127
  


\bibitem{Carrion}
R. Reyes-Carri\'on, \emph{A generalization of the notion of instanton}, Diff. Geom. Appl. (1998) 1-20.


\bibitem{Strom}
J. Strom, \emph{Modern Classical Homotopy Theory}, Graduate Studies in Mathematics, AMS, 2011.


\bibitem{Tanaka}
Y. Tanaka, \emph{A construction of $\Spin(7)$-instantons}, Ann. Global Anal. Geom. 42 (2012) 495-521.



\bibitem{Walpuski}
T. Walpuski, \emph{$\Spin(7)$-instantons, Cayley submanifolds and Fueter sections}, Comm. Math. Phys. (2014) 1-36.

\bibitem{Ward}
R. Ward, \emph{Completely solvable gauge-field equations in dimension greater than four}, Nuclear Phys. B. 236 (1984) 381--396.

\bibitem{Yokota}
I. Yokota, \emph{Exceptional Lie groups}, arXiv:0902.0431

\end{thebibliography}

\end{document}